\documentclass[11pt,leqno]{amsart}
\usepackage{amsmath,amssymb,amsxtra,color,calligra,mathrsfs,comment,url}

\usepackage{xcolor} 
\colorlet{mdtRed}{red!50!black}
\definecolor{dblue}{rgb}{0,0,.6}
\usepackage[colorlinks]{hyperref}
\hypersetup{linkcolor=blue,citecolor=dblue,filecolor=dullmagenta,urlcolor=mdtRed}

\usepackage[all]{xy}
\usepackage{tikz,tikz-cd,tkz-graph,enumerate}
\usetikzlibrary{matrix,arrows,decorations.pathmorphing}

\DeclareMathOperator{\rk}{\mathrm{rank}}

\DeclareMathOperator{\Spec}{\textnormal{Spec}}

\DeclareMathOperator{\GL}{\textnormal{GL}}

\DeclareMathOperator{\Supp}{\textnormal{Supp}}
\DeclareMathOperator{\Proj}{\textnormal{Proj}}
\DeclareMathOperator{\Hilb}{\mathcal{H}{\it ilb}}
\DeclareMathOperator{\Rep}{\textnormal{Rep}}

\DeclareMathOperator{\et}{\textnormal{\'et}}

\newcommand{\m}{\mathfrak{m}}
\newcommand{\mc}[1]{\mathcal{#1}}
\newcommand{\dv}{\vee\vee}
\newcommand{\C}{\mathbb{C}}

\numberwithin{equation}{subsection}

\newtheorem{theorem}[equation]{Theorem}
\newtheorem{corollary}[equation]{Corollary}
\newtheorem{lemma}[equation]{Lemma}
\newtheorem{proposition}[equation]{Proposition}
\newtheorem*{theorem-nonumber}{Theorem}
\newtheorem{definition}[equation]{Definition}

\theoremstyle{definition}

\makeatletter
\newcommand\thankssymb[1]{\textsuperscript{\@fnsymbol{#1}}}
\newcommand\thanksletter[1]{\lowercase{\textsuperscript{\@alph{#1}}}}

\makeatother

\newcommand{\corauth}[2][]{\thanks{#2Corresponding author}} 

\begin{document}

\title[On Fundamental Group Schemes]{Fundamental Group Schemes of \uppercase{H}ilbert Scheme of $n$ Points on a Smooth Projective Surface} 

\author[A. Paul]{Arjun Paul\thankssymb{1}}

\address[A. Paul]{Department of Mathematics, Indian Institute of Technology Bombay, Powai, Mumbai 400076, Maharashtra, India.} 

\email{arjun.math.tifr@gmail.com} 

\author[R. Sebastian]{Ronnie Sebastian} 

\address[R. Sebastian]{Department of Mathematics, Indian Institute of Technology Bombay, Powai, Mumbai 400076, Maharashtra, India.} 

\email{ronnie@math.iitb.ac.in} 

\corauth[]{\thankssymb{1}}

\subjclass[2010]{14J60, 14F35, 14L15, 14C05}

\keywords{Finite vector bundle, $S$-fundamental group-scheme, Hilbert scheme, semistable bundle, Tannakian category.} 

\begin{abstract}
	Let $k$ be an algebraically closed field of characteristic $p > 3$. 
	Let $X$ be an irreducible smooth projective surface over $k$. Fix an integer 
	$n \geq 2$ and let ${\mathcal{H}{\it ilb}}_X^n$ be the Hilbert scheme parameterizing 
	effective $0$-cycles of length $n$ on $X$. The aim of the present article is to find the 
	$S$-fundamental group scheme and Nori's fundamental group scheme of the Hilbert scheme 
	$\mathcal{H}{\it ilb}_X^n$. 
\end{abstract}

\maketitle

\section{Introduction}
Let $X$ be a connected, reduced and complete scheme over a perfect field $k$ 
and let $x \in X$ be a $k$-rational point.
In \cite{No1}, Nori introduced a $k$-group scheme $\pi^N(X,x)$ associated to essentially 
finite vector bundles on $X$, and in \cite{No2}, the definition of $\pi^N(X,x)$ 
was extended to connected and reduced $k$-schemes. 
In \cite{BPS}, Biswas, Parameswaran and Subramanian defined the notion of {\it $S$-fundamental group scheme} 
$\pi^S(X, x)$ for $X$ a smooth projective curve over any algebraically closed field $k$. 
This was generalized to higher dimensional connected smooth projective $k$-schemes and studied extensively by 
Langer in \cite{La, La2}.  
In general, $\pi^S(X, x)$ carries more information than $\pi^N(X, x)$ and $\pi^{\et}(X, x)$. 
There are natural faithfully flat homomorphisms of affine $k$-group schemes 
$\pi^S(X, x) \to \pi^N(X, x) \to \pi^{\et}(X, x)$. 
The reader is referred to the introductions in 
\cite{No2} and \cite{La} for more details. Precise definitions of the above 
objects are given in the next section.
It is an interesting problem to determine $\pi^{\et}(X, x)$, $\pi^N(X, x)$ and 
$\pi^S(X, x)$ for well-known varieties. 

Let $k$ be an algebraically closed field 
and let $n\geq 2$. Let $\Hilb_X^n$ be the Hilbert scheme of 
$n$ points on an irreducible smooth projective surface $X$ over $k$. It is known that $\Hilb_X^n$ is an 
irreducible smooth projective variety of dimension $2n$ over $k$. The geometry of  $\Hilb_X^n$ has been 
extensively studied, see \cite{Fo, FGAex, Ia} and the references therein. In \cite{Beauville-83}
the author computes the topological fundamental group of $\Hilb_X^n$
when $k=\mathbb C$. 
In \cite[Theorem 1.1, Theorem 1.2]{BH} Biswas and Hogadi
show that the \'etale fundamental group 
$\pi^{\et}(\Hilb_X^n, \widetilde{nx})$ is isomorphic to the abelianization of 
$\pi^{\et}(X, x)$, for any $x \in X(k)$. Here $\widetilde{nx}$ is a point 
in $\Hilb_X^n$ mapping to $nx\in S^n(X)$.
Therefore, it is natural to ask if a similar result holds for 
$\pi^N(\Hilb_X^n, \widetilde{nx})$ and $\pi^S(\Hilb_X^n, \widetilde{nx})$. 
In this paper we answer this question affirmatively when 
the base field $k$ has characteristic $p > 3$. 
The following is the main result of this article.

\begin{theorem-nonumber}[Theorem \ref{main-thm-S-fgs}]
Let ${\rm char}(k) > 3$ and $n\geq 2$. Then there is an isomorphism of affine $k$-group schemes
$$\widetilde{f} : \pi^S(X, x)_{\rm ab} \stackrel{\sim}{\longrightarrow} \pi^S(\Hilb_X^n, \widetilde{nx}).$$
In particular, $\pi^S(\Hilb_X^n, \widetilde{nx})$ is an abelian group scheme. 
\end{theorem-nonumber}

\noindent
From the above we can easily deduce the following result.

\begin{theorem-nonumber}[Theorem \ref{main-thm-N-fgs}]
Let ${\rm char}(k) > 3$ and $n\geq 2$. Then there is an isomorphism of affine $k$-group schemes
$$\widetilde{f}^? : \pi^?(X, x)_{\rm ab} \stackrel{\sim}{\longrightarrow} \pi^?(\Hilb_X^n, \widetilde{nx})\,,$$
where $? = N,\, \et$. As above, this shows that the groups $\pi^?(\Hilb_X^n, \widetilde{nx})$
are abelian for $?=N,\et$.
\end{theorem-nonumber}

The assertion about $\pi^{\et}(\Hilb_X^n,\widetilde{nx})$ is a corollary of the 
main result in \cite{BH}, which is proved using a different method.

We briefly describe the organization of this paper. 
In \S2 we recall the main definitions and results on fundamental group schemes 
that we need from \cite{No2} and \cite{La}. In \S3 we recall and prove results that we need 
about the Hilbert scheme and the Hilbert-Chow map. The main input in this paper 
is the construction in \S4, which we briefly explain here. 
Let $\varphi:\Hilb^n_X\to S^n(X)$ 
denote the Hilbert-Chow morphism and let $\psi:X^n\to S^n(X)$ denote the  
quotient map under the natural action of $S_n$ on $X^n$. Given a numerically flat 
sheaf $E$ on $\Hilb^n_X$, we can associate to it a coherent sheaf on $X^n$,  
namely, $\psi^*\varphi_*E$. However, it is not clear if this 
coherent sheaf is numerically flat. To remedy this, we associate to $E$ a locally free 
sheaf on a large open subset of $X^n$ and take its unique reflexive extension. 
Then we use the criterion \cite[Theorem 2.2]{La2} (this criterion is proved in 
\cite{La} but stated more precisely in \cite{La2}) to check that this reflexive sheaf 
is locally free. From this construction we are able to define a homomorphism 
$\pi^S(X,x)_{\rm ab} \to \pi^S(\Hilb_X^n, \widetilde{nx})$. In \S5, we use the criterion in 
\cite[Proposition 2.21]{DMOS} to show that this homomorphism is an isomorphism.

The hypothesis on the characteristic is needed in two places. First we use it 
in Proposition \ref{prop-1} to compute the power series ring at a certain closed point 
in $S^n(X)$; see equation \eqref{formal-nbd}. Here we need that ${\rm char}(k)\neq 2$. 
In Proposition \ref{closed-immersion} we show that the above homomorphism is a closed
immersion. Here we need that ${\rm char}(k)>3$, which allows
us to show that a certain sheaf on $S^n(X)$ is locally free on a very large open 
subset using Proposition \ref{invariant-pushforward-locally-free}.

We make a remark about the strategy of our proof. 
When $k=\C$, one easily checks that the map 
$\varphi_*$ induces an isomorphism of topological fundamental groups
$\pi_1(\Hilb_X^n,\widetilde{nx})\to \pi_1(S^n(X),nx)$, and then
one shows that $\pi_1(S^n(X),nx)$ is isomorphic to $\pi_1(X,x)_{\rm ab}$, see \cite[Lemma 1, page 767]{Beauville-83}.
In \cite{BH}, the authors show that there are isomorphisms 
$\varphi_*:\pi^{\et}(\Hilb_X^n,\widetilde{nx})\stackrel{\sim}{\to}\pi^{\et}(S^n(X),nx)$
and $\pi^{\et}(S^n(X),nx)\cong \pi^{\et}(X,x)_{\rm ab}$.
In this article, however, we do not make the intermediate comparison
with the group schemes corresponding to $S^n(X)$.

Note that $\Hilb_X^1 \cong X$ and so we always assume that $n \geq 2$. 

\subsection*{Acknowledgements} We thank Indranil Biswas for suggesting this 
question to us. We are very grateful to the referee for an extremely careful reading 
of this paper and for many useful suggestions. 

\section{Fundamental Group Schemes}
In the rest of this article, unless mentioned otherwise, $k$ will denote an algebraically closed field 
of characteristic $p > 3$. 

\subsection{Nori's fundamental group scheme}\label{N-fgs}
Let $X$ be a connected, proper and reduced $k$-scheme. 
We denote by ${\rm QCoh}(X)$ the category of quasi-coherent sheaves of $\mathcal{O}_X$-modules on $X$. 
Consider the full subcategory ${\rm Vect}(X)$ of ${\rm QCoh}(X)$, 
whose objects are locally 
free coherent sheaves of $\mathcal{O}_X$-modules  (vector bundles). 
A vector bundle $E$ is said to be {\it finite} if there are distinct non-zero polynomials 
$f, g \in \mathbb{Z}[t]$ with non-negative coefficients such that $f(E) \cong g(E)$. 

Let $C$ be a connected smooth projective curve over $k$. The {\it degree} of a vector bundle $E$ on $C$ is defined 
to be the number 
$$\deg(E) := c_1(E) \cdot [C]\,.$$ 
 A vector bundle $E$ on $C$ is said to be {\it semistable} 
if for any non-zero proper subbundle $F \subset E$, we have 
$$\mu(F) := \frac{\deg(F)}{\rk(F)} \leq \ \frac{\deg(E)}{\rk(E)} =: \mu(E)\,.$$ 

\begin{definition}\label{cat-nf}
	Let $X$ be a connected, projective and reduced $k$-scheme. Let $\mathcal{C}^{\rm nf}(X)$ denote the full subcategory of 
	${\rm QCoh}(X)$ whose objects are coherent sheaves $E$ on $X$ satisfying the following two conditions: 
	\begin{enumerate}
		\item $E$ is locally free, and 
		\item for any smooth projective curve $C$ over $k$ and any morphism $f : C \longrightarrow X$, 
		the vector bundle $f^*E$ is semistable of degree $0$. 
	\end{enumerate}
\end{definition}
We call the objects of the category $\mc C^{\rm nf}(X)$ 
{\it numerically flat vector bundles} on $X$. 
In the literature these are also referred to as Nori semistable vector bundles
and this category is also denoted by $NS(X)$; for example, see \cite{Es-Me}. 
See also \cite[Remark 5.2]{La}. 
However, we reserve the term semistable to refer to slope semistable.

\begin{definition}
	A vector bundle $E$ on $X$ is said to be {\it essentially finite} if there exist two 
	numerically flat vector bundles $V_1, V_2$ and 
	finitely many finite vector bundles $F_1, \ldots, F_n$ on $X$ with 
	$V_2 \subseteq V_1 \subseteq \bigoplus\limits_{i=1}^n F_i$ such that $E \cong V_1/V_2$. 
\end{definition}

Unless otherwise specified, for any coherent sheaf $E$ on $X$, we denote by $E_x$ the fiber of $E$ at $x \in X$. 
Let ${\rm EF}(X)$ be the full subcategory of ${\rm Vect}(X)$ whose objects are essentially finite vector bundles on $X$. 
Let ${\rm Vect}_k$ be the category of finite dimensional $k$-vector spaces. 
Fix a closed point $x \in X$ and let 
$$T_x : {\rm EF}(X) \longrightarrow {\rm Vect}_k$$
be the fiber functor defined by 
sending an object $E \in {\rm EF}(X)$ to its fiber $E_x$ at $x$. 
Then the quadruple $({\rm EF}(X), \bigotimes, T_x, \mathcal{O}_X)$ is a neutral Tannakian category. 
The affine $k$-group scheme $\pi^N(X, x)$ representing the functor of $k$-algebras $\underline{\rm Aut}^{\otimes}(T_x)$ 
is called {\it Nori's fundamental group scheme} of $X$ based at $x$ 
(see \cite[Section 1]{DMOS} for definition of the functor $\underline{\rm Aut}^{\otimes}(T_x)$). 
It is shown in \cite[Proposition 4, p.~88]{No2} that 
$\pi^N(X, x) \cong \pi^N(X, y)$ for any two closed points $x, y \in X$. 

\subsection{$S$-fundamental group scheme}\label{S-fgs}

Let $E$ be a coherent sheaf on $X$. Denote by $E^\vee$ the sheaf 
$\mathscr{H}\!{\it om}(E,\mc O_X)$. 
A coherent sheaf $E$ is said to be {\it reflexive} if the natural $\mathcal{O}_X$-module 
homomorphism $E \to E^{\dv}$ is an isomorphism. 
Let $X$ be a connected smooth projective variety over $k$ of dimension $d$. 
Let $H$ be an ample divisor on $X$. 
The {\it degree} of a torsion free coherent sheaf $E$ on $X$ is defined 
to be the number 
$$\deg(E) := c_1(E) \cdot H^{d-1}\,.$$ 
A sheaf $E$ on $X$ is said to be {\it $H$-semistable} 
if for any non-zero proper subsheaf $E' \subset E$, we have 
$$\mu(E') := \frac{\deg(E')}{\rk(E')} \leq \ \frac{\deg(E)}{\rk(E)} =: \mu(E)\,.$$ 
Let $F$ denote the absolute Frobenius morphism. We say that 
$E$ is {\it strongly $H$-semistable} its Frobenius pullbacks $(F^n)^*E$ are 
$H$-semistable, for all $n\geq 0$; see \cite[p.~252]{Langer-semistable.sheaves.positive.char.}.
\begin{definition}
	Let $X$ be a connected, smooth and projective variety over $k$ of dimension $d$ and let $H$ be an ample divisor on $X$. 
	Let ${\rm Vect}_0^{s}(X)$ be the full subcategory of ${\rm QCoh}(X)$ whose objects are coherent sheaves 
	$E$ on $X$ satisfying the following three conditions: 
	\begin{enumerate}
		\item $E$ is reflexive, 
		\item $E$ is strongly $H$-semistable, and 
		\item ${\rm ch}_1(E)\cdot H^{d-1} = {\rm ch}_2(E)\cdot H^{d-2} = 0$, 
		where ${\rm ch}_i(E)$ denote the $i$-th Chern character of 
		$E$, for all $i = 1, 2$. 
	\end{enumerate}
\end{definition}

Since $X$ is smooth, it follows from \cite[Proposition 4.1]{La} that the objects of ${\rm Vect}_0^s(X)$ 
are in fact locally free sheaves and all of their Chern classes vanishes. Moreover, the category 
${\rm Vect}_0^s(X)$ does not depend on the choice of ample divisor $H$ \cite[Proposition 4.5]{La}. 

Assume that $X$ is smooth. Fix a $k$-valued point $x \in X$. 
Let $T_x : {\rm Vect}_0^s(X) \longrightarrow {\rm Vect}_k$ be the fiber functor defined by 
sending an object $E$ of ${\rm Vect}_0^s(X)$ to its fiber $E_x \in {\rm Vect}_k$ at $x$. 
Then $({\rm Vect}_0^s(X), \otimes, T_x, \mathcal{O}_X)$ is a neutral Tannaka category 
\cite[Proposition 5.5, p.~2096]{La}. 
The affine $k$-group scheme $\pi^S(X, x)$ Tannaka dual to this category is called the 
{\it S-fundamental group scheme} of $X$ with base point $x$ \cite[Definition 6.1, p.~2097]{La}. 

The following result may be well-known to experts, but we could not find a precise
reference, so we include a proof. See also the proof of 
\cite[Chapter II, Proposition 4 (d), page 88]{No2}.

\begin{lemma}\label{invariance-of-base-point}
	Let $X$ be a connected, smooth and projective $k$-scheme. 
	Then $\pi^S(X, x_1) \cong \pi^S(X, x_2)$, for all $x_1, x_2 \in X(k)$.  
\end{lemma}

\begin{proof}
	Since $\pi^S(X, x)$ is the affine $k$-group scheme representing the functor of $k$-algebras 
	$\underline{\rm Aut}^{\otimes}(T_x)$, where $T_x$ is the fiber functor 
	$T_x : {\rm Vect}_0^s(X) \longrightarrow {\rm Vect}_k$, it suffices to show that, 
	for any two points $x_1, x_2 \in X(k)$,
	the fiber functors $T_{x_1}$ and $T_{x_2}$ are isomorphic. 
	Given any object $\mc V \in {\rm Vect}_0^s(X)$, we need to define a {\it natural} $k$-linear isomorphism 
	$$\eta_{\mc V} : T_{x_1}(\mc V) = {\mc V}_{x_1} \longrightarrow {\mc V}_{x_2} = T_{x_2}(\mc V)\,;$$ 
	meaning that for any morphism $f : \mc V \to \mc V'$ of objects in ${\rm Vect}_0^s(X)$, 
	the following diagram should commute. 
	\begin{equation}\label{functoriality}
	\xymatrix{
		T_{x_1}(\mc V) \ar[rr]^{T_x(f)} \ar[d]^{\eta_{\mc V}} && T_{x_1}(\mc V') \ar[d]^{\eta_{\mc V'}} \\ 
		T_{x_2}(\mc V) \ar[rr]^{T_y(f)} && T_{x_2}(\mc V') 
	}
	\end{equation}
	
	For any group scheme $H$ over $k$, denote by $\Rep_k(H)$ the category of 
	representations of $H$ into finite dimensional $k$-vector spaces. 
	Let $G = \pi^S(X, x_1)$. Then there is an equivalence of categories 
	$\zeta : {\rm Vect}_0^s(X) \stackrel{\sim}{\longrightarrow} {\rm Rep}_k(G)$ 
	and the inverse of this equivalence of categories defines a principal $G$-bundle $p : P \to X$, 
	(see \cite[Proposition 2.9]{No1} for the construction), 
	known as the {\it $S$-universal cover of $X$} (see \cite[p. 2097]{La}).  
	This associates to a $G$-module $V$ an object $\mc V := P \times^G V$ in the category ${\rm Vect}_0^s(X)$; 
	moreover, any morphism $\mc V \to \mc V'$ in the category ${\rm Vect}_0^s(X)$ comes from a $G$-module 
	homomorphism $V \to V'$ in ${\rm Rep}_k(G)$. 
	
	Fix two points $\tilde{x}_1, \,\tilde{x}_2 \in P$ such that $p(\tilde{x}_i) = x_i$, for $i = 1, 2$. 
	Then we have isomorphisms 
	\begin{equation*}
	\xi_i : G \longrightarrow P_{x_i}, \qquad \,\, i = 1, 2\,. 
	\end{equation*}
	Let $\rho : G \to \GL(V)$ be a finite dimensional linear representation and let 
	$\mc V := P \times^G V$ be the associated vector bundle on $X$. 
	Then we have $k$-linear isomorphisms 
	$$\widetilde{\xi_i} : \mc{V}_{x_i} = P_{x_i} \times^G V \stackrel{\xi_i^{-1}}{\cong} G \times^G V 
	\stackrel{\simeq}{\longrightarrow} V$$ 
	induced by $\xi_i$, for all $i = 1, 2$.  
	This gives a $k$-linear isomorphism of the fibers 
	$$\widetilde{\xi_2}^{-1}\circ\widetilde{\xi_1} =: \eta_{\mc V} : \mc{V}_{x_1} \longrightarrow \mc{V}_{x_2}\,.$$ 
	Since any homomorphism $f : \mc V \to \mc V'$ of objects in ${\rm Vect}_0^s(X)$ comes from 
	a $G$-module homomorphism $\tilde f : V \to V'$, it follows from above construction that the 
	above diagram in \eqref{functoriality} commutes. 
\end{proof}

\section{Hilbert-Chow Morphism}
\subsection{Hilbert scheme of length $n$ cycles}
From now on we denote by $X$ an irreducible smooth projective surface over $k$. 
For an integer $n \geq 2$, let $S_n$ be the permutation group 
of $n$ symbols. Then $S_n$ acts on the product $X^n$ and the associated quotient $S^n(X) = X^n/S_n$ is 
a normal projective variety of dimension $2n$ over $k$. Note that $S^n(X)$ is not smooth. Its smooth locus 
$S^n(X)_{\rm sm} \subset S^n(X)$ is the open dense subscheme consisting of reduced effective $0$-cycles of length $n$ in $X$. 
Since $\dim_k(X) = 2$, the singular locus $S^n(X)_{\rm sing} := S^n(X) \setminus S^n(X)_{\rm sm}$ is a closed 
subscheme of codimension $2$ in $S^n(X)$. 

Let $\Hilb_X^n$ be the Hilbert scheme parametrizing effective $0$-cycles of length $n$ in $X$. 
This is an irreducible smooth projective scheme of dimension $2n$ over $k$. Consider the Hilbert-Chow morphism 
\begin{equation}\label{chow-map}
	\varphi : \Hilb_X^n \longrightarrow S^n(X)\, ,
\end{equation}
given by sending $Z \in \Hilb_X^n$ to 
$$\sum\limits_{p \in \Supp(Z)} \ell(\mathcal{O}_{Z, p}) [p] \in S^n(X),$$ 
where 
$$\Supp(Z) = \{p \in X : \mathcal{O}_{Z, p} \neq 0 \}$$ 
denotes the support of the $0$-cycle $Z$ in $X$ and 
$\ell(\mathcal{O}_{Z, p})$ the length of the local ring $\mathcal{O}_{Z, p}$ as a module over itself. 
It is well known that $\varphi$ is a proper morphism.

\subsection{Stratification of $S^n(X)$}\label{stratification}
A point $y \in S^n(X)$ can be written as 
$$\sum\limits_{j=1}^r n_j x_j,$$ 
where $x_1, \ldots, x_r \in X$ are distinct 
points with multiplicities 
\begin{equation}\label{stratification-indices}
	n_1 \geq n_2 \geq \cdots \geq n_r \in \mathbb{Z}_{>0}, 
\end{equation}
respectively, such that 
$\sum_{j=1}^r n_j = n$. 
The $r$-tuple of positive integers 
$$\langle n_1, n_2, \ldots, n_r \rangle$$ 
is called the \textit{type} of $y$. 
Let $Z_{\langle n_1, n_2, \ldots, n_r\rangle}$ denote the locus of points in $S^n(X)$ of type 
$\langle n_1, n_2, \ldots, n_r\rangle$. The fiber $\varphi^{-1}(y)$ has dimension $n-r$, for all 
$y \in Z_{\langle n_1, n_2, \ldots, n_r\rangle}$ (see \cite[p.~667]{Fo}). 
The dimension of the locus of points of type $\langle n_1,n_2,\ldots,n_r\rangle$ is $2r$. 
From this the following lemma follows. 

\begin{lemma}\label{strata-dimension}
The dimension of the subset $\varphi^{-1}(Z_{\langle n_1, n_2, \ldots, n_r\rangle})$
is $n+r$.
\end{lemma}

\subsection{Fibers of Hilbert-Chow morphism}
Let $W \subset S^n(X)$ denote the open subset consisting of points of type 
$\langle 1,1,1,\ldots,1\rangle$ and $\langle 2,1,1,\ldots,1\rangle$.
Let $V$ denote the open subset $\varphi^{-1}(W)$ and let 
\begin{equation}\label{eqn-phi'}
	\varphi : V \longrightarrow W 
\end{equation}
be the restriction of the morphism $\varphi$ in \eqref{chow-map} to $V$. 
It follows from Lemma \ref{strata-dimension} that the dimension of $\Hilb_X^n\setminus V$ is $n+n-2 = 2n-2$ 
and hence ${\rm codim}_{\Hilb_X^n}(\Hilb_X^n \setminus V) = 2$. 

It was shown in \cite[Lemma 4.3, p.~668]{Fo} that for any point $q \in S^n(X)$ of type $\langle 2, 1, 1, \ldots, 1\rangle$, 
the schematic fiber $\varphi^{-1}(q)$, with its reduced structure, is isomorphic to $\mathbb{P}_k^1$. 
We need that ${\varphi}^{-1}(q)$ is reduced. 
We could not find a precise reference for Proposition \ref{prop-1}, 
which is well known to experts, so we include a proof. 

First we recall the following result. 
\begin{lemma}\label{lem-1}
	Let $I$ be an ideal of a commutative ring $A$ with identity. 
	Let $A[It] := \bigoplus\limits_{i=0}^{\infty} I^i t^i \subset A[t]$ be the Rees algebra of $I$ in the polynomial ring $A[t]$. 
	Let $\pi : \Proj(A[It]) \to \Spec A$ be the associated projective $A$-scheme. 
	For an $A$-algebra $B$, consider the graded $A$-algebra structure on $A[It] \otimes_A B$ given by 
	$(A[It] \otimes_A B)_d := (I^d \otimes_A B) t^d$, for all $d \geq 0$. 
	Then we have a canonical isomorphism of $A$-schemes 
	\begin{equation*}
	\psi : \Proj(A[It] \otimes_A B) \stackrel{\simeq}{\longrightarrow} \Proj(A[It]) \times_{\Spec A} \Spec B\,. 
	\end{equation*}
\end{lemma}

\begin{proof}
	Follows from \cite[Lemma 26.11.6., \href{https://stacks.math.columbia.edu/tag/01MX}{Tag 01MX}]{Stk}. 
\end{proof}

\begin{proposition}\label{prop-1}
Assume that ${\rm char}(k) \neq 2$. 
Let $q \in W$ be a point of type $\langle 2, 1, 1, \ldots, 1\rangle$.
The scheme theoretic fiber $\varphi^{-1}(q)$ is a reduced subscheme of $V$. 
\end{proposition}

\begin{proof}
Let $\tilde{q} \in X^n$ be a point such that $\tilde{q} \mapsto q$ under the 
natural map $\psi : X^n \to S^n(X)$. The formal neighbourhood of $\tilde{q}$ is given by the spectrum of the local ring 
$$\widehat{\mc O}_{X^n, \tilde q} = k[[x_1, y_1, x_2, y_2, \ldots, x_n, y_n]]\,.$$ 
There is an inclusion $\widehat{\mc O}_{W, q} \hookrightarrow \widehat{\mc O}_{X^n, \tilde q}$. 
By the discussion in the paragraph just before \cite[Theorem 7.3.4, p.~170]{FGAex}, we have 
\begin{equation}\label{formal-nbd}
\widehat{\mc O}_{W, q} = k[[u, v, w, x', y', x_3, y_3, \ldots, x_n, y_n]]/(uw - v^2)\,, 
\end{equation}
where 
$x = x_1 - x_2$, $y = y_1 - y_2$, $x' = x_1 + x_2$, $y' = y_1 + y_2$, $u = x^2$, $v = xy$ and $w = y^2$. 
Here we are using the assumption ${\rm char}(k) \neq 2$. 

Let $Z \subset W$ denote the irreducible closed subset consisting of points of type 
$\langle 2, 1, 1, \ldots, 1\rangle$. Let $J$ denote the stalk at $q$ of the ideal sheaf of $Z$ in the local ring 
$\mc O_{W, q}$ and let $\widehat{J}$ denote its image in $\widehat{\mc O}_{W, q}$. 
Now $Z$ is contained in the image $\psi(X^{n-1}) \subset S^n(X)$, where the inclusion $X^{n-1} \hookrightarrow X^n$ 
is given by 
$$(x, x_3, x_4, \ldots, x_n)\, \longmapsto\, (x, x, x_3, \ldots, x_n)\,.$$ 
Clearly, the ideal of $X^{n-1}$ in $\widehat{\mc O}_{X^n, \tilde q}$ is given by $x_1 - x_2 = y_1 - y_2 = 0$. 
From this, we conclude that $\widehat{J}$ is the kernel of the composite homomorphism 
$$\widehat{\mc O}_{W, q}\hookrightarrow \widehat{\mc O}_{X^n, \tilde{q}}\, 
\twoheadrightarrow\, \widehat{\mc O}_{X^n, \tilde q}/(x, y)\,,$$ 
where $x = x_1 - x_2$ and $y = y_1 - y_2$. 
This proves that $\widehat{J} = (u, v, w)$. 

By \cite[Lemma 4.4]{Fo} the map $\varphi$ is the blowup of 
$W$ along $Z$. Let $\mc O_{W, q}[tJ]$ denote the Rees algebra 
of the ideal $J$. By Lemma \ref{lem-1}, the schematic fiber $\varphi^{-1}(q)$ is 
$${\rm Proj}\ (\mc O_{W, q}[tJ]) \times_{\Spec (\mc O_{W, q})} \Spec(\mc O_{W, q}/\m_q) 
\cong {\rm Proj}\ (\mc O_{W, q}[tJ] \otimes_{\mc O_{W, q}} (\mc O_{W, q}/\m_q))\,,$$ 
where $\m_q$ is the maximal ideal of the local ring $\mathcal{O}_{W, q}$ at $q$. 
It follows from the isomorphism 
\begin{align*}
	\mc O_{W, q}[tJ]\otimes_{\mc O_{W, q}} (\mc O_{W, q}/\m_q) & \cong \mc O_{W, q}[tJ] \otimes_{\mc O_{W, q}} 
	(\widehat{\mc O}_{W, q}/\widehat{\m}_q)\\ 
	& \cong \widehat{\mc O}_{W, q}[t\widehat J] \otimes_{\widehat{\mc O}_{W, q}} (\widehat{\mc O}_{W, q}/\widehat{\m}_q) 
\end{align*}
that the schematic fiber $\varphi^{-1}(q)$ is 
$$
{\rm Proj}\ (\widehat{\mc O}_{W, q}[t\widehat J]\otimes_{\widehat{\mc O}_{W, q}}(\widehat{\mc O}_{W, q}/\widehat{\m}_q)) \cong 
{\rm Proj}\ (\widehat{\mc O}_{W, q}[t\widehat J]) \times_{{\rm Spec}(\widehat{\mc O}_{W, q})}
{\rm Spec}(\widehat{\mc O}_{W, q}/\widehat{\m}_q)\,.
$$ 

\noindent 
Write 
$$A:=\widehat{\mc O}_{W, q} = k[[u, v, w, x', y', x_3, y_3, \ldots, x_n, y_n]]/(uw - v^2)\,. $$ 
It is clear that the maximal ideal $\m:=\widehat{\m}_q$ of $A$ is given by 
$$\m=(u, v, w, x', y', x_3, y_3, \ldots, x_n, y_n)A.$$
First let us understand the scheme ${\rm Proj}(A[t\widehat{J}])$. 
This scheme is covered by affine open subsets given by ${\rm Spec}$ of the following three affine $k$-algebras: 
\begin{eqnarray}
	R_1 &:=& \Big(k[[u, v, w, x', y', x_3, y_3, \ldots, x_n, y_n]]/(uw - v^2)\Big)\left[\dfrac{v}{u},\dfrac{w}{u}\right],\nonumber\\
	R_2 &:=& \Big(k[[u, v, w, x', y', x_3, y_3, \ldots, x_n, y_n]]/(uw - v^2)\Big)\left[\dfrac{u}{v},\dfrac{w}{v}\right],\nonumber\\
	R_3 &:=& \Big(k[[u, v, w, x', y', x_3, y_3, \ldots, x_n, y_n]]/(uw - v^2)\Big)\left[\dfrac{u}{w},\dfrac{v}{w}\right]\nonumber .
\end{eqnarray}
Let us first consider the ring $R_1$. In this ring, $\frac{w}{u}=\frac{wu}{u^2}=\frac{v^2}{u^2}$. Therefore,
we get that 
$$R_1=\Big(k[[u, v, w, x', y', x_3, y_3, \ldots, x_n, y_n]]/(uw - v^2)\Big)\left[\dfrac{v}{u}\right].$$
Similarly, since $\frac{u}{w} = \frac{uw}{w^2}=\frac{v^2}{w^2}$, we get that 
$$R_3=\Big(k[[u, v, w, x', y', x_3, y_3, \ldots, x_n, y_n]]/(uw - v^2)\Big)\left[\dfrac{v}{w}\right].$$
Further, in $R_2$ we have $\frac{u}{v}\frac{w}{v}=1$. Therefore, 
$$R_2=\Big(k[[u, v, w, x', y', x_3, y_3, \ldots, x_n, y_n]]/(uw - v^2)\Big)\left[\dfrac{v}{u},\dfrac{u}{v}\right].$$
It is now clear that the scheme ${\rm Proj}(A[t\widehat{J}])$ is covered by ${\rm Spec}\, R_1$ 
and ${\rm Spec}\,R_3$, since ${\rm Spec}\,R_2$ is an open subset of each of these. 
Now we need to compute 
$$R_1\otimes_A (A/\m)\qquad \textnormal{and} \qquad R_3\otimes_A (A/\m).$$ 
Let us first write 
$$R_1=A[T]/(Tu-v).$$ 
Now note that $(Tu-v)\subset \m A[T]$ since $\m$ contains $u$ and $v$. Therefore we get 
\begin{eqnarray}
R_1\otimes_A (A/\m) & \cong & R_1/\m R_1 \nonumber \\
& \cong & A[T]/\m A[T] \cong k[T] \nonumber. 
\end{eqnarray}
Similarly, we have $R_3\otimes_A (A/\m) \cong k[T]$. 
Thus we have proved that the scheme theoretic fiber $\varphi^{-1}(q)$ is reduced and is isomorphic to $\mathbb{P}^1_k$. 
\end{proof}

\section{Homomorphism of $S$-fundamental Group Schemes}
Fix a closed point $x\in X$ and let $\widetilde{nx}$ be a point 
in $\Hilb_X^n$ mapping to $nx\in S^n(X)$.
In this section we construct a homomorphism of $S$-fundamental group schemes 
$$\pi^S(X,x)_{\rm ab} \longrightarrow \pi^S(\Hilb_X^n,\widetilde{nx})\,,$$
where $\pi^S(X,x)_{\rm ab}$ is the abelianization of $\pi^S(X,x)$. 

\subsection{A group scheme theoretic lemma}
We need the following group scheme theoretic result for later use. 
First recall the definition of the derived subgroup $\mathscr DG$
as given in \cite[\S 10.1]{Wa}. It is a closed normal subgroup. 
It follows from the main Theorem in \cite[\S16.3]{Wa} that there is a 
quotient $\alpha:G\to G_{\rm ab}$ whose kernel is precisely $\mathscr DG$.
It is clear that $G_{\rm ab}$ is an abelian affine group scheme. 
\begin{lemma}\label{group-lemma}
	Let $G$ and $H$ be two group schemes over $k$. For an integer $n \geq 2$, we denote by $G^n$ the 
	group scheme $G \times \cdots \times G$ (= the $n$-fold product of $G$ with itself). 
	Then $S_n$ acts on $G^n$ by permuting the factors. 
	Let $f_0$ be the following composite group homomorphism 
	$$f_0 : G^n \stackrel{\alpha^n}{\longrightarrow} (G_{\rm ab})^n \stackrel{m}{\longrightarrow} G_{\rm ab}\,,$$ 
	where $m$ denotes the 
	multiplication homomorphism. 
	Then a homomorphism of $k$-group schemes $f : G^n \longrightarrow H$ is $S_n$-invariant if and only if there 
	is a homomorphism $\tilde f : G_{\rm ab} \longrightarrow H$ of affine $k$-group schemes such that 
	$\tilde f \circ f_0 = f$. In other words, the following diagram commutes. 
	\[
	\xymatrix{
		G^{n}\ar[rr]^f\ar[rd]_{f_0} && H\\
		& G_{\rm ab}\ar[ru]_{\tilde f} & 
	}
	\]
\end{lemma}

\begin{proof}
	For any $k$-group scheme $G$, we denote by 
	\begin{itemize}
		\item $m_G : G \times G \longrightarrow G$ the multiplication morphism of $G$, 
		\item $i_G : G \longrightarrow G$ the inversion morphism of $G$, and 
		\item $e_G \in G(k)$ the identity of $G$. 
	\end{itemize}
	We sketch the proof for $n = 2$; the general case is similar and left to the reader as an exercise. 
	We have a homomorphism $f : G \times G \longrightarrow H$ such that $f \circ \sigma = f$, where 
	$\sigma : G \times G \to G \times G$ is the homomorphism switching the factors. 
	Let $p_1, p_2 : G \times G \to G$ denote the projections onto the first and second factors, respectively. 
	Then one can easily check that 
	\begin{align*}
	f\circ (m_G,e_G)&=f\circ m_{G\times G}\circ ((p_1,e),(p_2,e))\\
	&=m_H\circ (f\circ(p_1,e),f\circ (p_2,e))\\
	&=m_H\circ (f\circ(p_1,e),f\circ (e,p_2))\\
	&=f\circ (p_1,p_2)\\
	&=f\circ(m_G\circ \sigma,e_G)\,.
	\end{align*}
	Using this it easily follows that 
	\begin{equation}
	f\circ m_{G\times G}((m_G,e_G),(i_G\circ m_G\circ\sigma,e_G))=e_H\,.
	\end{equation}
	Now one easily concludes that $f$ factors through 
	the map $G\times G\to G_{\rm ab}\times G_{\rm ab}$.
	Let $\Delta':G\to G\times G$ denote the map $g\mapsto (g,g^{-1})$.
	Then one checks easily that $f\circ \Delta'=e_H$. From these the 
	lemma follows. 
\end{proof}

A vector bundle $E$ on $X^n$ is said to be {\it $S_n$-invariant} if $\sigma^*E \cong E$, 
for all $\sigma \in S_n \subseteq {\rm Aut}(X^n)$. 

\begin{corollary}\label{S_n-equivariance-str}
	Any vector bundle in the category ${\rm Vect}_0^s(X^n)$, associated to a representation of 
	$\pi^S(X, x)^n$ which factors through $\pi^S(X, x)_{\rm ab}$ (see the statement of Lemma 
	\ref{group-lemma}), is $S_n$-invariant. 
\end{corollary}

\subsection{A functor between Tannakian categories}
Given a numerically flat vector bundle $E$ on $\Hilb_X^n$, we want to 
associate to it a numerically flat vector bundle $\mathcal G$ on $X^n$. 
We first associate to $E$ a reflexive sheaf 
$\mathcal G$ on $X^n$ and then use the criterion in \cite[Theorem 2.2]{La2}
to show that $\mathcal G$ is numerically flat. 
We recall the criterion here for the benefit of the reader. 
\begin{theorem-nonumber}\cite[Theorem 2.2]{La2}
Let $X$ be a smooth projective $k$-variety of dimension $d$. Let $H$
be an ample divisor on $X$ and let $E$ be a coherent sheaf on $X$.
Then $E$ is numerically flat if and only if 
$E$ is a strongly $H$-semistable reflexive sheaf
with ${\rm ch}_1(E)\cdot H^{d-1}={\rm ch}_2(E)\cdot H^{d-2}=0$.
\end{theorem-nonumber}

Recall that $W \subset S^n(X)$ is the open subset consisting of points of type 
$\langle 1,1,1,\ldots,1\rangle$ and $\langle 2,1,1,\ldots,1\rangle$, and
$V$ is the open subset $\varphi^{-1}(W) \subset \Hilb_X^n$, where $\varphi$ is the Hilbert-Chow 
morphism.

\begin{proposition}\label{pushforward-locally-free-large-open} 
Let $E$ be a numerically flat vector bundle of rank $r$ on $\Hilb_X^n$. Then $\varphi_*(E\vert_V)$ is a 
locally free coherent sheaf on $W$. Moreover, the natural map 
\begin{equation}\label{eqn-8}
	\varphi^*\varphi_*(E\vert_V) \longrightarrow E\vert_V 
\end{equation}
is an isomorphism.
\end{proposition}

\begin{proof}
Let $q\in W$ be a point of type $\langle 2,1,1,\ldots,1\rangle$. 
Let $\mathcal{I} \subset \mathcal{O}_V$ denote the reduced sheaf of ideals 
of the closed subscheme $\varphi^{-1}(q)$. Let $\mathscr{I}_q$ 
be the ideal sheaf of the closed point $q \in W$. 
For each integer $n \geq 1$, let $\mathscr{I}_q^n$ be 
the ideal sheaf of the $n$-th order thickening of $q$ in $W$. 
By Proposition \ref{prop-1} we have 
$$\mc I=\mathscr{I}_q\mc O_V.$$ 
For each integer $n \geq 1$, let $Y_n$ denote the closed subscheme 
of $V$ corresponding to the sheaf of ideals $\mathcal{I}^n$. 
Since $E$ is numerically flat and $Y_1 \cong \mathbb{P}^1_k$ (see Proposition \ref{prop-1}), it follows that 
the restriction of $E$ to $Y_1$ is trivial. 

Consider the following short exact sequence of sheaves on $V$
\begin{equation}\label{eqn-4}
	0\longrightarrow \mc I\otimes E\longrightarrow E\longrightarrow E\vert_{Y_1}\longrightarrow 0. 
\end{equation}
Applying $\varphi_*$ to it we get the following exact sequence of sheaves on $W$. 
\begin{equation}\label{eqn-5}
	\varphi_*(E)\longrightarrow H^0(Y_1,E\vert_{Y_1})\longrightarrow R^1\varphi_*(\mc I\otimes E). 
\end{equation}

We claim that the completion of $R^1\varphi_*(\mc I\otimes E)$ at the maximal ideal $\m_q$ of $q$ is $0$. 
By the Theorem on Formal Functions (see \cite[Chapter III, Theorem 11.1]{Ha}), we have 
\begin{equation}\label{eqn-6}
	(R^1\varphi_*(\mc I\otimes E))\ \widehat{\,}\ \cong \lim_{\longleftarrow}H^1(Y_n,\mc I\otimes E\otimes \mc O_V/\mc I^n). 
\end{equation}
We will prove by induction on $n$ that $H^1(Y_n,\mc I\otimes E\otimes \mc O_V/\mc I^n)=0$.
Since $\mc I = \mathscr{I}_q \mathcal{O}_V$, it follows that there is a surjection 
$$(\m_q^n/\m_q^{n+1})\otimes_{\mc O_{W,q}}\mc O_{V}\cong 
	\mathscr{I}^n_q/\mathscr{I}^{n+1}_q\otimes_{\mc O_W}\mc O_V\twoheadrightarrow \mc I^n/\mc I^{n+1},$$ 
where $\mc O_{W, q}$ is the stalk of $\mc O_W$ at $q$. 
The locally free sheaf $\mc I^n/\mc I^{n+1}$ on $Y_1 \cong \mathbb P^1$ is a 
direct sum of line bundles. It follows that each of these line bundle has
degree $\geq 0$. 
For $n = 1$, the base case of induction, we have 
\[
H^1(Y_1,\mc I\otimes E\otimes \mc O_V/\mc I)=H^1(Y_1,\mc I/\mc I^2\otimes E_1)=0\,.
\]
Assume that we have proved the assertion for $n$. Then the assertion for $n+1$ follows from the 
long exact cohomology sequence attached to the short exact sequence of sheaves on $Y_{n+1}$
\[
0 \longrightarrow (\mc I^{n+1}/\mc I^{n+2})\otimes E \longrightarrow (\mc I/\mc I^{n+2})\otimes E 
	\longrightarrow (\mc I/\mc I^{n+1})\otimes E \longrightarrow 0\,.
\]
This proves the claim that $R^1\varphi_*(\mc I\otimes E)$ at the maximal ideal $\m_q$ of $q$ is $0$. 

This proves that the natural map 
\begin{equation}\label{eqn-7}
	\varphi_*(E) \longrightarrow H^0(Y_1,E\vert_{Y_1}) 
\end{equation}
in \eqref{eqn-5} is surjective in a neighborhood around $q$. Let $s_1,s_2\ldots,s_r$ 
be a basis for $H^0(Y_1,E\vert_{Y_1})$. Let $\Spec(A)$ be an affine neighborhood of $q$
on which the map in \eqref{eqn-7} is surjective. Choosing lifts 
$\tilde s_i\in \Gamma(\Spec(A), \varphi_*(E))$ of $s_i$, we get a homomorphism 
\begin{equation}
	\mc O_V^{\oplus r} \longrightarrow E 
\end{equation}
over $\varphi^{-1}(\Spec(A))$, which is a surjection over the fiber $Y_1$. Since $\varphi$ is proper,
it follows that there is a smaller affine neighborhood $W_0$ of $q$ over 
which there is an isomorphism $\mc O_{V_0}^{\oplus r}\stackrel{\sim}{\to} E$, where 
$V_0=\varphi^{-1}(W_0)$. Applying $\varphi_*$, using normality 
of $S^n(X)$ and that $\varphi$ is birational, the Proposition follows. 
\end{proof}

\begin{corollary}\label{compatibility-Frobenius-pullback}
Let $F$ denote the absolute Frobenius morphism. With the above notations, we have an isomorphism 
$F^*\varphi_*(E\vert_V) \stackrel{\sim}{\longrightarrow} \varphi_*(F^*E\vert_V)$.
\end{corollary}

\begin{proof}
Since $F^*E$ is numerically flat, it follows that both these sheaves are locally free 
of the same rank. It suffices to show that the natural map 
\begin{equation}\label{Frob}
	F^*\varphi_*(E\vert_V) \longrightarrow \varphi_*(F^*E\vert_V) 
\end{equation}
is surjective. This is clear over the smooth locus of $S^n(X)$ since $F$ is faithfully
flat over the smooth locus. Let $q \in W$ be a point of type $\langle 2, 1, 1 \ldots, 1 \rangle$. 
It follows from Proposition \ref{prop-1} that the restriction of $F^*\varphi_*(E\vert_V)$ to $q$ 
is naturally isomorphic to $H^0(Y_1, E_1)$ and the restriction of $\varphi_*(F^*E\vert_V)$ 
at $q$ is naturally isomorphic to $H^0(Y_1, F^*E_1)$. The restriction to $q$ of the natural homomorphism 
in \eqref{Frob} is the map 
$$F^* : H^0(Y_1, E_1) \longrightarrow H^0(Y_1, F^*E_1)\,,$$ 
which is a surjection. From this the Corollary follows. 
\end{proof}

Recall the quotient map $\psi : X^n \longrightarrow S^n(X)$ defined in \eqref{chow-map}.  
Let $j : \psi^{-1}(W) \hookrightarrow X^n$ denote the inclusion. Recall that the category 
$\mc C^{\rm nf}(X)$ is defined in Definition \ref{cat-nf}. 

\begin{proposition}\label{main-correspondence}
	If $E$ is an object of $\mc C^{\rm nf}(\Hilb_X^n)$, then 
	$$\mathscr G(E) := (j_*\psi^*\varphi_*(E\vert_V))^{\dv}$$ 
	is an object of $\mc C^{\rm nf}(X^n)$. 
\end{proposition}

\begin{proof}
It is proved in Proposition \ref{pushforward-locally-free-large-open} that $\varphi_*(E)$ is locally free on $W$. 
Since $X^n\setminus \psi^{-1}(W)$ has codimension $\geq 4$, it follows that 
\begin{equation}\label{bundle-G}
	\mathscr{G}(E) := (j_*\psi^*\varphi_*(E\vert_V))^{\vee\vee} 
\end{equation} 
is a coherent reflexive sheaf on $X^n$. 
For notational simplicity, we denote by $\mc G$ the sheaf $\mathscr G(E)$. 
Note that $\mc G\vert_{\psi^{-1}(W)} = \psi^*\varphi_*(E\vert_V)$ is locally free. 

Choose $m\gg0$ so that $mH$ is very ample. 
Choose general hyperplanes $H_1,\ldots,H_{d-1}\in \vert mH\vert$
so that $C = H_1 \cap H_2 \cap \cdots \cap H_{d-1} \stackrel{i}{\hookrightarrow} \psi^{-1}(W)$ 
is a smooth complete intersection curve whose image $\psi(C)$ lies in the smooth locus of $S^n(X)$. 
Since $\varphi:\varphi^{-1}(S^n(X)_{\rm sm})\to S^n(X)_{\rm sm}$ 
is an isomorphism, we can lift $i$ to a morphism $\tilde{i}$ which 
makes the following diagram commute.
\[
\xymatrix{
	&& V \ar[d]^{\varphi} \ar@{^(->}[r] & \Hilb_X^n \ar[d]^\varphi \\
	C \ar@{^(->}[r]^-i\ar@/^1pc/[urr]^-{\tilde i} & \psi^{-1}(W) \ar[r]^\psi & W\ar@{^(->}[r] & S^n(X)  
}
\]
It follows from Proposition \ref{pushforward-locally-free-large-open} that 
$$i^*\mc G\cong \tilde i^*\varphi^*\varphi_*(E\vert_V) \cong \tilde i^*(E\vert_V)\,.$$
Since $E$ is in $\mc C^{\rm nf}(\Hilb_X^n)$ it follows that $i^*\mc G$ is semistable
of degree $0$. This shows that $\mc G$ is $H$-semistable. 

In Corollary \ref{compatibility-Frobenius-pullback} we proved that the locally free sheaves  
$F^*\varphi_*(E\vert_V)$ and $\varphi_*(F^*E\vert_V)$ are isomorphic. 
Since $X^n$ is smooth the Frobenius is faithfully flat and so $F^*\mc G$ is reflexive 
(use the characterization that a coherent module $M$ over a local ring $A$ is 
reflexive iff it sits in a short exact sequence $0\to M\to A^{\oplus r}\to A^{\oplus s})$. 
The restriction of $F^*\mc{G}$ on $\psi^{-1}(W)$ is 
$$F^*\psi^*\varphi_*(E\vert_V)\cong \psi^*F^*\varphi_*(E\vert_V)\cong \psi^*\varphi_*(F^*E\vert_V).$$
Since the reflexive extension on $X^n$ is unique (see \cite[Proposition 1.6, p.~126]{Ha2}), we conclude that 
$$F^*\mc G\cong (j_*(\psi^* \varphi_*(F^*E\vert_V)))^{\vee\vee}\,.$$ 
Since $E \in \mc C^{\rm nf}(\Hilb_X^n)$ we have $F^*E\in \mc C^{\rm nf}(\Hilb_X^n)$; 
then following the arguments in the preceding paragraph, we see that $F^*\mc G$ 
is $H$-semistable. In this way we can show that all Frobenius pullbacks 
of $\mc G$ are semistable. This shows that $\mc G$ is strongly $H$-semistable.

It is clear from above that ${\rm ch}_1(\mc G)\cdot H^{d-1} = 0$. 
Choose general hyperplanes $H_1, \ldots, H_{d-2}$ in the linear system $\vert mH\vert$ so that 
$$S = H_1 \cap H_2 \cap \cdots \cap H_{d-2} \subset \psi^{-1}(W)$$ 
is a smooth surface. We can do this since $X^n\setminus \psi^{-1}(W)$ has codimension $\geq 4$. 
It suffices to show that ${\rm ch}_2(\mc G\vert_S) = 0$. 
Now $\mc G\vert_S$ is locally free as $S\subset \psi^{-1}(W)$ and $\mc G$ 
is locally free on $\psi^{-1}(W)$. Therefore, in view of \cite[Theorem 2.2]{La2}, it suffices to 
show that $\mc G\vert_S\in \mc C^{\rm nf}(S)$. But this follows from the arguments as in the second paragraph of this proof. 
Therefore, we have $\mc G\in {\rm Vect}_0^s(X^n)$ and hence by \cite[Theorem 2.2]{La2} 
$\mc G$ is locally free and is in $\mc C^{\rm nf}(X^n)$. This proves the proposition.
\end{proof}

\begin{proposition}\label{additive-functor}
With the above notations, 
$$\mathscr G : \mc C^{\rm nf}(\Hilb_X^n) \longrightarrow \mc C^{\rm nf}(X^n)$$ 
is a additive tensor functor. 
\end{proposition}

\begin{proof}
	First we show that $\mathscr G$ is a functor. 
	Let $f : E \to E'$ be a morphism in the category $\mc C^{\rm nf}(\Hilb_X^n)$. 
	We need to find a canonical morphism $\mathscr G(f) : \mathscr G(E) \to \mathscr G(E')$ in $\mc C^{\rm nf}(X^n)$. 
	There is a morphism $\psi^*\varphi_*(f):\mathscr G(E)\vert_{\psi^{-1}(W)}\to \mathscr G(E')\vert_{\psi^{-1}(W)}$.
	Since $X^n\setminus \psi^{-1}(W)$ has codimension $\geq 4$ and 
	$\mathscr G(E)$, $\mathscr G(E')$ are locally free, it follows that this 
	morphism extends uniquely to give a morphism $\mathscr G(E)\to \mathscr G(E')$. 
	
	The bundles $\mathscr G(E\oplus E')$ and $\mathscr G(E)\oplus \mathscr G(E')$ are naturally isomorphic
	on $\psi^{-1}(W)$ and so they are naturally isomorphic. Similarly, 
	$\mathscr G(E\otimes E')$ is naturally isomorphic to 
	$\mathscr G(E)\otimes \mathscr G(E')$.
\end{proof}

\subsection{Homomorphism of group schemes}
Fix distinct $k$-valued points $x_1, \ldots, x_n \in X(k)$ of $X$. 
Let $\tilde{x} \in \Hilb_X^n(k)$ be such that $\varphi(\widetilde{x}) = \psi(x_1, \cdots, x_n)\in S^n(X)_{\rm sm}$.
For any locally free sheaf $E$ on $\Hilb_X^n$, there are natural isomorphisms 
of fibers
$$E_{\tilde x}\cong (\varphi_*E)_{\varphi(\tilde x)}\cong (\psi^*\varphi_*(E))_{(x_1,x_2,\ldots,x_n)}.$$
Consider the following diagram. 
$$
\xymatrix{
	(\mc C^{\rm nf}(\Hilb_X^n), \otimes, T_{\tilde x}, \mathcal{O}_{\Hilb_X^n})\ar[r] & 
							(\mc C^{\rm nf}(X^n), \otimes, T_{(x_1, \ldots, x_n)}, \mathcal{O}_{X^n})\ar[d]\\
	(\mc C^{\rm nf}(\Hilb_X^n), \otimes, T_{\widetilde{nx}}, \mathcal{O}_{\Hilb_X^n})\ar[u] & 
							(\mc C^{\rm nf}(X^n), \otimes, T_{(x, \ldots, x)}, \mathcal{O}_{X^n})
	}
$$
The horizontal arrow is a morphism of Tannakian categories due 
to Propositions \ref{main-correspondence} and \ref{additive-functor}. The two vertical arrows 
are due to Lemma \ref{invariance-of-base-point}. Thus, we get a 
homomorphism of $S$-fundamental group schemes 
\begin{equation}
	f : \pi^S(X^n, (x, \ldots, x)) \longrightarrow \pi^S(\Hilb_X^n, \widetilde{nx})\,. \nonumber
\end{equation}
For $\sigma\in S_n$ we get an automorphism $\sigma_*$ of  
$\pi^S(X^n,(x,\ldots,x))$. It is easily checked that $f\circ\sigma_*=f$. 
By \cite[Theorem 4.1, p.~842]{La2} there is an isomorphism
	\begin{equation}
	\pi^S(X^n, (x, \ldots, x)) \stackrel{\sim}{\longrightarrow} \pi^S(X, x) \times_k \cdots \times_k \pi^S(X, x). \nonumber
	\end{equation}
By abuse of notation, denote the composite of $f$ and the inverse 
of this isomorphism by $f$. Thus, we have a homomorphism 
\begin{equation}\label{S-group-homomorphism}
	f : \pi^S(X, x) \times_k \cdots \times_k \pi^S(X, x) \longrightarrow \pi^S(\Hilb_X^n, \widetilde{nx})
\end{equation}
which satisfies $f\circ \sigma_* = f$. It follows from Lemma \ref{group-lemma}
that the homomorphism 
of the $S$-fundamental group schemes in \eqref{S-group-homomorphism} factors through a homomorphism 
\begin{equation}\label{Pi-S-ab-homomorphism}
	\tilde{f} : \pi^S(X, x)_{\rm ab} \longrightarrow \pi^S(\Hilb_X^n, \widetilde{nx})\,. 
\end{equation}
This completes the construction of our homomorphism of $k$-group schemes.

\section{Isomorphism of Group Schemes}
In this section we use \cite[Proposition 2.21]{DMOS} to show that 
the homomorphism $\tilde f$ in \eqref{Pi-S-ab-homomorphism} is an isomorphism. 

\subsection{$S_n$-invariant line bundles}
We begin with a discussion on why a numerically
flat $S_n$-invariant line bundle on $X^n$ descends to a
line bundle on $S^n(X)$. A more general
result is proved in \cite[Proposition 3.6]{Fo-77}. 
For the benefit of the authors and the 
reader we include a proof of the statement that we need.

\begin{proposition}\label{descent-line-bundles}
Let $\mc L$ be a numerically flat $S_n$-invariant line bundle on $X^n$. 
Then there is a numerically flat line bundle $L'$ on $S^n(X)$ such that $\psi^*L'=\mc L$.
\end{proposition}

\begin{proof}
The assertion that $L'$, if it exists, is numerically flat follows easily. 
We now prove its existence.

Let ${\rm Pic}^\tau$ denote the subscheme of the Picard scheme 
whose closed points parametrize numerically trivial line bundles. 
By \cite[Corollary 4.7]{La2} we have 
\[({\rm Pic}^\tau(X^n))_{\rm red}=\prod_{i=1}^n({\rm Pic}^\tau(X))_{\rm red}.\]
Thus, there is a numerically trivial line bundle $L_0$ on $X$ 
such that $\mc L=\bigotimes\limits_{i=1}^np_i^*L_0$.

The rest of the proof is a more detailed version of the first 
part of the proof in \cite[Proposition 3.6]{Fo-77}.
Let $H\subset S_n$ denote the subgroup ${\rm Stab}(1)$.
Let $D\subset X$ be an ample divisor such that $L_0$ is trivial
on $U=X\setminus D$. Let $s\in \Gamma(U,L_0)$ be a global 
section which generates $L_0$ over $U$. Then $p_1^*s$ is a generating section 
of $p_1^*L_0$ over the open subset $U\times X\times \ldots\times X$ 
and this section is invariant under the action of the subgroup 
$H$. In particular, the section $p_1^*s$ also generates the 
line bundle $p_1^*L_0$ over the smaller open subset $U^{n}$ 
and is invariant under the action of the subgroup $H$.

Given $(x_1,x_2,\ldots,x_n)\in X^n$, let $D$ be an ample divisor 
in $X$ which does not contain any of the $x_i$. If $U=X\setminus D$, then it 
is clear that $(x_1,\ldots,x_n)$ is in the $S_n$-invariant open subset 
$U^{n}$. Thus, we can cover $X^n$ by open subsets of this type.
Using this observation, we can find 
a finite collection of ample divisors $D_\alpha\subset X$ 
(set $U_\alpha=X\setminus D_\alpha$) and sections $s_\alpha\in \Gamma(U_\alpha,L_0)$ 
such that 
\begin{enumerate}
	\item $t_\alpha:=p_1^*s_\alpha$ generates $p_1^*L_0$ on the open subset 
	$U_\alpha^{n}$,
	\item $t_\alpha$ is invariant under $H$, and
	\item $X^n=\bigcup_\alpha U_\alpha^{ n}$. 
\end{enumerate}
Define functions $f_{\alpha\beta}\in \mc O_X(U_\alpha^n\cap U_\beta^n)^{\times}$ by 
$$t_\alpha=f_{\alpha\beta}t_\beta.$$
It follows that $f_{\alpha\beta}$ are invariant under $H$.
Let $\sigma_i:=(1,i)$ for $1\leq i\leq n$ be left coset representatives 
of $H$ in $G$. The functions $\prod_{i=1}^n\sigma_i^*(f_{\alpha\beta})$
are clearly invariant under $S_n$ and satisfy the cocyle condition. 
Let $V_\alpha\subset S^n(X)$ be $\psi(U_\alpha^{n})$. 
It is clear that $V_\alpha$ is open and $\psi^{-1}(V_\alpha)=U_\alpha^{ n}$.
Thus, using the above cocycle we get a line bundle on $S^n(X)$ which 
is trivial on $V_\alpha$. It is clear that the pullback of this line bundle 
is isomorphic to $\bigotimes_{i=1}^np_i^*L_0$, which completes the proof
of the proposition.
\end{proof}

\subsection{Faithfully flatness}
In this subsection we use to show that 
the homomorphism $\tilde f$ in \eqref{Pi-S-ab-homomorphism} is 
faithfully flat. 
We begin by recalling \cite[Proposition 2.21]{DMOS}  
for the convenience of the reader. 

Let $\theta : G \longrightarrow G'$ 
be a homomorphism of affine group schemes over $k$ and let 
\begin{equation}\label{eqn-hom-f}
	\widetilde{\theta} : \Rep_k(G') \longrightarrow \Rep_k(G) 
\end{equation}
be the functor given by sending $\rho' : G' \to \GL(V)$ to $\rho'\circ \theta : G \to \GL(V)$. 
An object $\rho : G \to \GL(V)$ in $\Rep_k(G)$ is said to be a {\it subquotient} of an object 
$\eta : G \to \GL(W)$ in $\Rep_k(G)$ if there are two $G$-submodules $V_1 \subset V_2$ of $W$ 
such that $V \cong V_2/V_1$ as $G$-modules. 

\begin{proposition}[Proposition 2.21, \cite{DMOS}]\label{DM}
	Let $\theta : G \longrightarrow G'$ be a homomorphism of affine algebraic groups over $k$. Then 
	\begin{enumerate}[(a)]
		\item $\theta$ is faithfully flat if and only if the functor $\widetilde{\theta}$ in \eqref{eqn-hom-f} 
		is fully faithful and given any subobject $W \subset \widetilde{\theta}(V')$, with $V' \in \Rep_k(G')$, 
		there is a subobject $W' \subset V'$ in $\Rep_k(G')$ such that $\widetilde{\theta}(W') \cong W$ in $\Rep_k(G)$. 
		
		\item $f$ is a closed immersion if and only if every object of $\Rep_k(G)$ 
		is isomorphic to a subquotient of an object of the form $\widetilde{\theta}(V')$, 
		for some $V' \in \Rep_k(G')$. 
	\end{enumerate}
\end{proposition}

\begin{proposition}\label{faithfully-flatness}
	The homomorphism $$\tilde f:\pi^S(X, x)_{\rm ab}\to \pi^S(\Hilb_X^n, \widetilde{nx})$$ 
	defined in \eqref{Pi-S-ab-homomorphism} is faithfully flat.
\end{proposition}

\begin{proof}
We will use Proposition \ref{DM} (a). Let $E_1$ be an object in 
the category ${\rm Vect}_0^s(\Hilb_X^n) = \mathcal{C}^{\rm nf}(\Hilb_X^n)$. 
Let $\mathcal{G}_1:=\mathscr{G}(E_1)$ be the vector bundle as defined in \eqref{bundle-G}. 
Clearly $\mc{G}_1$ has the same rank as that of $E_1$. 
If $\mc{G}_2 \subset \mc{G}_1$ is a subbundle corresponding to a representation of $\pi^S(X, x)_{\rm ab}$, 
we need to show that there is a subbundle $E_2 \subset E_1$ such that $\mc{G}_2=\mathscr{G}(E_2)$. 
We will prove this by induction on the rank of $E_1$. If ${\rm rank}(E_1) = 1$, there is nothing to prove. 
Assume that ${\rm rank}(E_1) \geq 2$. 

The vector bundles $\mc G_i$ correspond to representations 
$$\pi^S(X^n, (x, \ldots, x)) \stackrel{f_0}{\longrightarrow} 
\pi^S(X, x)_{\rm ab} \stackrel{\rho_i}{\longrightarrow} \GL(V_i).$$ 
Since $\pi^S(X, x)_{\rm ab}$ is an abelian affine $k$-group scheme, 
it follows from \cite[Theorem 9.4, p.~70]{Wa} that every irreducible representation
of it is one dimensional. From this one easily checks that the $\pi^S(X,x)_{\rm ab}$-module
$V_1/V_2$ will have a one dimensional quotient. Thus, there is a one dimensional quotient 
$V_1\to L_1$ such that $V_2$ is a $\pi^S(X,x)_{\rm ab}$-submodule of the kernel of this homomorphism. 
Let $\mc L$ be the line bundle on $X^n$ corresponding to the representation $L_1$. 
Then it is clear that $\mc L$ is $S_n$-invariant (see Corollary \ref{S_n-equivariance-str}) 
and there is an $S_n$-equivariant exact sequence of bundles 
$$0 \longrightarrow \mc{K} \longrightarrow \mc{G}_1 \longrightarrow \mc{L} \longrightarrow 0$$ 
on $X^n$ such that $\mc{G}_2 \subset \mc{K}$. 

It follows from Proposition \ref{descent-line-bundles} 
that $L' := (\psi_*\mc L)^{S_n}$ is a locally free line bundle on all of $S^n(X)$
and satisfies $\psi^*L'=\mc L$. 
Let $L:=\varphi^*L'$,  then it is easy to check that
$L$ is numerically flat on $\Hilb_X^n$. 

We claim that the following complex of sheaves on $W$ 
\begin{equation}\label{eqn-3}
0 \to (\psi_*\mc K)^{S_n}\Big\vert_W \to (\psi_*\mc G_1)^{S_n}\Big\vert_W \to (\psi_*\mc L)^{S_n}\Big\vert_W \to 0 
\end{equation}
is exact. The sequence \eqref{eqn-3} can fail to be exact only on the right. 
Note that $\psi_*(\mc G_1)^{S_n}$ restricted to $W$ is $\varphi_*(E_1\vert_V)$. 
Let $J$ be the cokernel: 
$$\varphi_*(E_1\vert_V)\to L'\Big\vert_W\to J\to 0\,.$$
Pulling this back by $\psi$ we get the following commutative diagram on $\psi^{-1}(W)$ with exact rows. 
\[\xymatrix{
	\psi^*\varphi_*(E_1\vert_V)\ar[r]\ar@{=}[d]& \psi^*L'\Big\vert_{\psi^{-1}(W)}\ar[r]\ar@{=}[d] & \psi^*J\ar[r] & 0\\
	\mc G_1\Big\vert_{\psi^{-1}(W)}\ar[r] & \mc L\Big\vert_{\psi^{-1}(W)}\ar[r] & 0 &
 }
\]
This shows that $\psi^*J=0$. It is easy to conclude that $J = 0$, since $\psi$ is surjective. 
This proves the exactness of \eqref{eqn-3}. It follows that $K' := (\psi_*\mc K)^{S_n}$ is locally free on $W$. 
Applying $\varphi^*$ to \eqref{eqn-3}, we get the following short exact sequence of locally free sheaves on $V$. 
\[
	0 \longrightarrow (\varphi^*K')\vert_V \longrightarrow E_1\vert_V \longrightarrow L\vert_V \longrightarrow 0\,. 
\] 
Since both $E_1$ and $L$ are locally free on a smooth variety
and $\Hilb_X^n\setminus V$ has codimension $\geq2$, it follows 
that this morphism on $V$ extends to a morphism $E_1\to L$. 
This being a nonzero morphism of numerically flat vector bundles 
and $L$ being of rank one, it follows that $E_1\to L$ is surjective.

It is clear that on $X^n$ we have $\mathscr G(L)=\mc L$. 
Let $K$ denote the kernel of the homomorphism $E_1 \longrightarrow L$. 
It is clear that $\mathscr G(K)=\mc K$. 
Since $\mc G_2 \subset \mc K$ the assertion that there is $E_2\subset E_1$ such that 
$\mc G_2=\mathscr{G}(E_2)$ follows by induction on rank.

To complete the proof of the proposition we need to show that if 
$E_1$ and $E_2$ are numerically flat vector bundles on $\Hilb_X^n$ then the natural map 
$${\rm Hom}_{\Hilb_X^n}(E_1,E_2) \stackrel{}{\longrightarrow} {\rm Hom}_{X^n}(\mc G_1,\mc G_2)$$ 
is bijective. It is clear that this natural map is injective (faithful). 
Therefore, it suffices to show the following. 
If $\mc G=\mathscr{G}(E)$, where $E$ is a numerically flat vector bundle on $\Hilb_X^n$, 
then any nonzero homomorphism $\phi : \mc O_{X^n} \longrightarrow \mc G$ comes from a nonzero homomorphism 
$\widetilde{\phi} : \mc O_{\Hilb_X^n} \longrightarrow E$. 
Since the homomorphism $\pi^S(X^n, x) \longrightarrow \pi^S(X, x)_{\rm ab}$ is faithfully flat, 
and $\mc G$ arises from a representation of $\pi^S(X, x)_{\rm ab}$, it follows that $\phi$ is a map between
two representations of $\pi^S(X, x)_{\rm ab}$. This shows that $\phi$ 
is $S_n$-equivariant on $X^n$.
Now from the preceding discussion it follows that $\phi$ arises
from a morphism $\mc O_{\Hilb_X^n}\longrightarrow E$.
\end{proof}

\subsection{Closed immersion}
In this subsection we show that the homomorphism $\tilde f$ in 
\eqref{Pi-S-ab-homomorphism} is a closed immersion. For this, we will apply 
\ref{DM} (b).

Let $q \in S^n(X)$ be a point of type $\langle n_1, n_2, \ldots, n_r \rangle$. 
Let $\tilde q_i$, for $i = 1, 2, \ldots, m$, denote the points in the fiber $\psi^{-1}(q)$. 
The stabilizer of $\tilde q_i$, denoted ${\rm St}(\tilde q_i)$, is isomorphic to 
$S_{n_1}\times S_{n_2}\times \ldots\times S_{n_r}$. Let $A$ denote the local ring 
$\mc O_{S^n(X),q}$ and let $B$ denote the semilocal ring $\mc 
O_{X^n}\otimes_{\mc O_{S^n(X)}} A$. 
Then $B$ is a finite $A$ module and $A = B^{S_n}$. 

Let $M$ be a $B$-module such that the action of $S_n$ on $B$ 
lifts to an action of $S_n$ on $M$. There is a short exact sequence
of $A$ modules 
\[
	0 \to M^{S_n} \to M \to \bigoplus_{g \in S_n} M \,, 
\]
where the last map is given by $m\mapsto (g\cdot m-m)_{g\in S_n}$. 
Let $\widehat{A}$ be the completion of $A$ with respect to its maximal ideal. 
Applying the functor $- \otimes_A\widehat{A}$, we conclude that
the following natural map is an isomorphism. 
$$\widehat{M^{S_n}} \stackrel{\sim}{\longrightarrow} \widehat{M}^{S_n}\,.$$ 
The ring $\widehat{B} = B \otimes_A \widehat{A}$ decomposes as 
\begin{equation}\label{decomp-of-B}
	\widehat{B} \cong \bigoplus_{i=1}^m \widehat{B}_i\,, 
\end{equation} 
where $\widehat{B}_i$ denotes the completion of $B$ at the maximal ideal 
corresponding to the point $\tilde q_i$, for all $i = 1, \ldots, m$. 
Applying the functor $M \otimes_B -$ to the above isomorphism \eqref{decomp-of-B}
we see that 
\begin{equation}\label{decomp-of-M-hat}
\widehat{M} \cong \bigoplus_{i=1}^m \widehat{M}_i\,, 
\end{equation}
where $M_i$ is the localization of $M$ at the maximal ideal corresponding
to the point $\tilde q_i$. Taking $S_n$-invariants in 
\eqref{decomp-of-M-hat}, it easily follows that 
$$\widehat{M}^{S_n} \cong \widehat{M}_i^{{\rm St}(\tilde q_i)}\,, \,\qquad \forall\,\, i\,.$$

\begin{proposition}\label{S_n-invariants-surjection}
	With notation as above, whenever ${\rm char}(k) > n_1$, 
	any $S_n$-equivariant surjective $B$-module homomorphism 
	$f : M \longrightarrow N$ of finitely generated $B$-modules descends to surjective 
	$A$-module homomorphism of their $S_n$-invariants $M^{S_n} \longrightarrow N^{S_n}$. 
\end{proposition}

\begin{proof}
	Suppose we have an $S_n$-equivariant exact sequence of $B$-modules 
	$$M \longrightarrow N \longrightarrow 0\,.$$ 
	Taking $S_n$-invariants we get a homomorphism of $A$-modules 
	\begin{equation}\label{map-of-S_n-invariants-M-to-N}
	M^{S_n} \longrightarrow N^{S_n}\,. 
	\end{equation} 
	To check this is surjective, it suffices to check that the map \eqref{map-of-S_n-invariants-M-to-N} 
	is surjective after passing to the completion. From the preceding discussion, it follows 
	that it suffices to check that 
	\begin{equation}\label{eqn-9}
	\widehat{M}_i^{{\rm St}(\tilde q_i)} \longrightarrow \widehat{N}_i^{{\rm St}(\tilde q_i)} 
	\end{equation}
	is surjective for one (and hence any) $i$. We know that $\widehat{M} \to \widehat{N}$ is surjective. 
	Thus, the above map in \eqref{eqn-9} will be surjective if we can lift a section of 
	$\widehat{N}_i^{{\rm St}(\tilde q_i)}$ to $\widehat{M}$ and average it, that is, apply 
	the operator 
	$$\frac{1}{\#{\rm St}(\tilde q_i)}\sum_{g\in {\rm St}(\tilde q_i)}g.$$
	This is possible if ${\rm char}(k) = p > n_1$ (c.f. inequalities \eqref{stratification-indices}). 
\end{proof}

\begin{proposition}\label{invariant-pushforward-locally-free}
Let $\mc G$ be a numerically flat $S_n$-invariant locally free sheaf on $X^n$. 
\begin{enumerate}[(i)]
	\item Let $q \in S^n(X)$ be a point of type $\langle n_1, n_2, \ldots, n_r \rangle$. 
	Assume that ${\rm char}(k) = p > n_1$. Then the sheaf $(\psi_*\mc G)^{S_n}$ is 
	locally free in a neighborhood of $q$. 
	
	\item Let $U_0$ denote the largest open subset where $(\psi_*\mc G)^{S_n}$ is locally free. 
	Then on $\psi^{-1}(U_0)$ the natural homomorphism 
	\begin{equation}\label{natural-map}
		\psi^*((\psi_*\mc G)^{S_n}) \longrightarrow \mc G 
	\end{equation}
	is an isomorphism.
\end{enumerate}
\end{proposition}

\begin{proof}
If $\mc G$ has rank 1 then $(\psi_*\mc G)^{S_n}$ is locally free on 
all of $S^n(X)$ and of rank one, see Proposition \ref{descent-line-bundles}. 
Since $\mc G$ corresponds to a representation of an 
abelian group scheme, it follows that there is an $S_n$-equivariant exact 
sequence of locally free sheaves on $X^n$ 
$$0 \longrightarrow \mc{K} \longrightarrow \mc G \longrightarrow \mc L \longrightarrow 0\,,$$ 
with ${\rm rank}(\mc{L}) = 1$. By induction on rank of $\mc G$, it suffices to show 
that the homomorphism on the right of the following exact sequence 
$$0 \longrightarrow (\psi_*\mc K)^{S_n} \longrightarrow (\psi_*\mc G)^{S_n} 
\longrightarrow (\psi_*\mc L)^{S_n}$$ 
is surjective in a neighbourhood of $q$. 
This surjection can be checked after passing to a formal neighbourhood of $q$. 
Now the first assertion of the Proposition follows from the above Proposition \ref{S_n-invariants-surjection}. 

To prove the second assertion, note that both sheaves are locally free of the same 
rank over $\psi^{-1}(U_0)$. The locus where the natural homomorphism \eqref{natural-map} 
is not an isomorphism is either empty or a closed subset of codimension $1$ in $\psi^{-1}(U_0)$. However, we know 
that the morphism $\psi$ is finite \'etale over the smooth locus of $S^n(X)$, hence the 
homomorphism \eqref{natural-map} is an isomorphism on the inverse image of the smooth 
locus of $S^n(X)$. Since the complement of the smooth locus of $S^n(X)$ has codimension 
$2$, it follows that the natural map in \eqref{natural-map} is an isomorphism over $\psi^{-1}(U_0)$. 
\end{proof}

\begin{lemma}\label{remove-J}
Let $T\subset S^n(X)$ be open. If $\delta: E_1\to E_2$  is a morphism between locally free sheaves on 
$T$, such that $\psi^*\delta$ is 
an isomorphism on $\psi^{-1}(T)$, then $\delta$ is an isomorphism. 
\end{lemma}
\begin{proof}
For a locally free sheaf $E$ on $T$, we have $E\cong [\psi_*(\psi^*E)]^{S_n}$. 
Thus, if $\delta: E_1\to E_2$  is a morphism on $T$, such that $\psi^*\delta$ is 
an isomorphism on $\psi^{-1}(T)$, then taking pushforward and $S_n$ invariants, 
it follows that $\delta$ is an isomorphism. 
\end{proof}

\begin{proposition}\label{closed-immersion}
	Let ${\rm char}(k) > 3$. Then the homomorphism $\tilde f$ in \eqref{Pi-S-ab-homomorphism} is a closed immersion. 
\end{proposition}

\begin{proof}
By Proposition \ref{DM} (b) it suffices to show that every 
$S_n$-invariant numerically flat bundle $\mc G$ on $X^n$ arises in the 
way described in Proposition \ref{main-correspondence}. In other words, 
we have to show that there is a numerically flat bundle $E$ on $\Hilb_X^n$ such that 
$\mc G = \big(j_*(\psi^*\varphi_*(E\vert_V))\big)^{\dv}$. 

Let $T \supset W$ be the open subset of $S^n(X)$ containing $W$ and points of  
type $\langle 3, 1, 1, \ldots, 1 \rangle$ and $\langle 2, 2, 1, \ldots, 1 \rangle$. 
Then $\varphi^{-1}(T)$ is an open subset of $\Hilb_X^n$ such that  
$\Hilb_X^n\setminus \varphi^{-1}(T)$ has codimension at least $3$ in $\Hilb_X^n$. 
Let $i : \varphi^{-1}(T) \hookrightarrow \Hilb_X^n$ denote the inclusion. 
Define 
$$E := (i_*\varphi^*(((\psi_*\mc G)\vert_T)^{S_n}))^{\dv}.$$ 
By Proposition \ref{invariant-pushforward-locally-free} we see that 
$(\psi_*\mc G)^{S_n}$ is locally free on $T$ and on 
$\psi^{-1}(T)$ the natural homomorphism $\psi^*\big((\psi_*\mc G)^{S_n}\big) \to \mc G$ is 
an isomorphism. Consider the natural homomorphisms 
\begin{equation}\label{composite-map}
	F^*\left((\psi_*\mc G)^{S_n}\right) \longrightarrow \left(F^*\psi_*(\mc G)\right)^{S_n} 
	\longrightarrow \left(\psi_*(F^*\mc G)\right)^{S_n}, 
\end{equation}
where $F$ denotes the absolute Frobenius morphism. 
We claim that the above composite homomorphism is an isomorphism over $T$. 
Applying $\psi^*$ to the above exact sequence, we get the following commutative diagram. 
\[
 \xymatrix{
	\psi^*F^*\left((\psi_*\mc G)^{S_n}\right) \ar[r] \ar[d]^{\wr} & \psi^*\left(\psi_*(F^*\mc G)\right)^{S_n} 
	\ar[d]^{\wr} \\ 
	F^*\mc G \ar@{=}[r] & F^*\mc G & & 
	}
\]
The two vertical arrows are isomorphisms on $T$ because of 
Proposition \ref{invariant-pushforward-locally-free}. It follows from Lemma  \ref{remove-J}
that the composite homomorphism in \eqref{composite-map} is an isomorphism over $T$. 
It follows that $F^*E \cong \left(i_*\varphi^*(\psi_*(F^*\mc{G}\vert_T)^{S_n})\right)^{\dv}$. 
Now imitating the proof of Proposition \ref{main-correspondence}
we see that $E$ is locally free and numerically flat on $\Hilb_X^n$.
It is clear that $\mathscr G(E)=\mc G$ (see the construction in 
the proof of Proposition \ref{main-correspondence}). 
This proves the Proposition. 
\end{proof}

\begin{theorem}\label{main-thm-S-fgs}
Let ${\rm char}(k) > 3$. Then the homomorphism 
$$\widetilde{f} : \pi^S(X, x)_{\rm ab} \longrightarrow \pi^S(\Hilb_X^n, \widetilde{nx})$$ in 
\eqref{Pi-S-ab-homomorphism} is an isomorphism. 
\end{theorem}

\begin{proof}
	Since $\widetilde{f}$ is faithfully flat by Proposition \ref{faithfully-flatness} and closed immersion 
	by Proposition \ref{closed-immersion}, it is an isomorphism. 
\end{proof}

From the above theorem we may easily conclude the following.

\begin{theorem}\label{main-thm-N-fgs}
	Let ${\rm char}(k) > 3$. There is an isomorphism of affine $k$-group schemes 
	$$\tilde{f}^? : \pi^?(X, x)_{\rm ab} \longrightarrow \pi^?(\Hilb_X^n, \widetilde{nx})\,,$$
	where $?=N,\et$. 
\end{theorem}
Let $E$ be an essentially finite vector bundle over a connected,
reduced and proper $k$-scheme $X$. Then there is a finite $k$-group
scheme $G$, a principal $G$-bundle $p : P \to X$ and a finite 
dimensional $k$-linear representation $\rho : G \to \GL(V')$ such that 
$E$ is the vector bundle associated to the representation $\rho$. 
It follows from the proof of \cite[Proposition 3.8]{No1} that there 
is a finite vector bundle $\mc V$ on $X$ such that $E$ is a subbundle of $\mc V$. 
It is clear that the functor $\mathscr G$ in Proposition \ref{main-correspondence}
takes a finite vector bundle to a finite vector bundle. It easily follows 
that $\mathscr G$ takes essentially finite vector bundles to essentially finite 
vector bundles. With these remarks we leave the details of the proof of Theorem \ref{main-thm-N-fgs}
to the reader. 

\vspace{1cm}
\noindent
\textbf{Declaration of competing interest.} \\ 
The authors declare that they have no conflict of interest.

\newcommand{\etalchar}[1]{$^{#1}$}

\end{document}